\documentclass[11pt]{amsart}
\usepackage{amsmath, amsthm, amssymb}
\usepackage[all]{xy}
\usepackage{mathtools}
\usepackage{enumitem}
\usepackage{mathrsfs}
\usepackage{tikz}
\usepackage{skak}
\usepackage{graphics}
\usepackage{array}
\usepackage{kbordermatrix}

\newcommand{\N}{\mathbb{N}}

\newcommand{\mcH}{\mathrel{\mathcal{H}}}
\newcommand{\mcR}{\mathrel{\mathcal{R}}}
\newcommand{\mcL}{\mathrel{\mathcal{L}}}
\newcommand{\mcD}{\mathrel{\mathcal{D}}}

\newcommand{\mcO}{\mathcal{O}}

\newcommand{\up}{\uparrow\!}
\newcommand{\down}{\downarrow\!}

\DeclareMathOperator{\dom}{dom}

\newtheorem{prop}{Proposition}[section]
\newtheorem{thm}[prop]{Theorem}
\newtheorem{cor}[prop]{Corollary}
\newtheorem{lem}[prop]{Lemma}
\theoremstyle{definition}

\newtheorem{rem}[prop]{Remark}

\newlist{thmenum}{enumerate}{10}
\setlist[thmenum,1]{label=\textnormal{(\alph*)}}
\setlist[thmenum,2]{label=\textnormal{(\roman*)}}

\begin{document}

\title{On the Inverse Hull of a Markov Shift}

\author{Aria Beaupr{\'e}}
\address{Harvey Mudd College\\
301 Platt Blvd\\
Claremont, CA 91711}
\email{abeaupre@g.hmc.edu}

\author{Anthony Dickson}
\address{Youngstown State University\\
1 University Plaza\\
Youngstown, OH 44555}
\email{ajdickson@student.ysu.edu}

\author{David Milan}
\address{Department of Mathematics\\
The University of Texas at Tyler\\
3900 University Boulevard\\
Tyler, TX 75799}
\email{dmilan@uttyler.edu}

\author{Christin Sum}
\address{California State University-Long Beach\\
1250 Bellflower Blvd\\
Long Beach, CA 90840}

\email{sum.christin@gmail.com}

\thanks{This research was supported by an NSF grant (DMS-1659221).}

\date{\today}
\subjclass[2010]{20M18, 37B10}

\begin{abstract} In this paper we provide an abstract characterization of the inverse hulls of semigroups associated with Markov shifts. As an application of the characterization we give an example of Markov shifts that are not conjugate, but have isomorphic inverse hulls. \end{abstract}

\maketitle

\section{Introduction}

We study inverse semigroups associated with Markov shift spaces. In \cite{StarlingShifts}, Starling defines an inverse semigroup associated with a one-sided subshift. He shows that the Carlsen-Matsumoto $C^*$-algebra associated with the subshift is generated by the inverse semigroup. He also gives a decomposition of the $C^*$-algebra as a partial crossed product by a free group. Recently, Exel and Steinberg \cite{ExelSteinberg} defined the inverse semigroups $H(S)$, associated with any $0$-left cancellative semigroup $S$, called the inverse hull of $S$. This collection is significant as it contains many of the examples of inverse semigroups that appear in the study of $C^*$-algebras generated by partial isometries. Graph inverse semigroups are included in this collection as well as semigroups associated with left cancellative categories. Also, the inverse semigroups Starling associates to one-sided shifts are inverse hulls.

One major challenge in working with the inverse hull of a $0$-left cancellative semigroup is to obtain a useful description of the semilattice of idempotents. Because $H(S)$ is defined to be the inverse semigroup generated by a set of partial bijections, it can be difficult to determine the possible sets, called \emph{constructible sets} in \cite{ExelSteinberg}, on which the idempotents of $H(S)$ act as identities. In section 3, we give a thorough description of the semilattice of the inverse semigroup associated with a Markov subshift. There is a set $\mcO$ of mutually orthogonal idempotents such that every idempotent is comparable to some element in $\mcO$. The idempotents strictly above $\mcO$ (together with 0) form a subsemigroup of the semilattice, with the additional property that each such idempotent is determined uniquely by the idempotents in $\mcO$ above which it sits. We describe a number of additional properties of this subsemigroup.

In section 4 we state a set of axioms on an inverse semigroup $H$ that are equivalent to $H$ being isomorphic to the inverse hull of the language of a Markov shift. We find that it is possible to axiomatize the set $\mcO$ appearing in such an inverse semigroup, although there is not necessarily a unique choice for $\mcO$. In section 5, we take advantage of this lack of uniqueness when we use our characterization to show that different Markov shifts can give rise to isomorphic inverse hulls. In particular we find two Markov shifts with different entropies and isomorphic inverse hulls. We also give an example of two conjugate Markov shifts whose inverse hulls are not isomorphic. Finally, we state a conjecture that if the inverse hulls of two Markov shifts are isomorphic, then their alphabets must have the same size.

\section{Preliminaries}
An \textit{inverse semigroup} is a semigroup $S$ such that for each $s$ in $S$, there 
exists a unique $s^*$ in $S$ such that 
\[
	s = s s^* s\quad \text{and}\quad s^* = s^* s s^*.
\]
The elements $e \in S$ satisfying $e^2 = e$ (and hence $e^* = e$) are called \emph{idempotents}. The set of all idempotents in $S$ is denoted by $E(S)$.

There is a natural partial order on $S$ defined by $s \leq t$ if $s = te$ for some idempotent $e$. Note that the subsemigroup $E(S)$ of idempotents is commutative, and hence forms a meet semilattice under the natural partial order with $e \wedge f= ef$ for $e,f$ in $E(S)$.

There are a number of useful relations known as Green's relations defined on a semigroup. For inverse semigroups, these relations take the following form: we have $s \mcL t$ if and only if $s^*s = t^*t$, $s \mcR t$ if and only if $ss^* = tt^*$, and $\mcH \,=\, \mcL \cap \mcR$. If $S$ has the property that for all $s,t$ in $S$, $s \mcH t$ implies $s = t$, then $S$ is said to be \emph{combinatorial}. One can also prove that $S$ is combinatorial provided $s^* s = s s^*$ implies $s$ is idempotent for all $s$ in $S$. Finally, $s \mcD t$ if and only if there exists $a,b \in S$ such that $a^*a = t^*t$, $aa^* = s^*s$, $b^*b = tt^*$, $bb^* = ss^*$, and $t = b^* s a$. For $e,f \in E(S)$ we have $e \mcD f$ if and only if there exists $a \in S$ with $e = a^*a$ and $f = aa^*$. 

An important example of an inverse semigroup is the semigroup $I(X)$ of partial bijections on a set $X$. A \emph{partial bijection} on $X$ is a bijection $g: A \to B$ with $A,B \subseteq X$. If $g \in I(X)$ with domain $A$ and range $B$ and $f \in I(X)$ with domain $C$ and range $D$, then the product $fg$ is the composition of the functions on the largest possible domain. That is, $fg$ is the bijection of $g^{-1}(B \cap C)$ onto $f(B \cap C)$. The map with empty domain is denoted by $0$. The inverse of $f$ in $I(X)$ is given by $f^{-1}$.

We will now outline the construction of the inverse hull of a left cancellative semigroup $S$ with zero. For more details, see \cite{ExelSteinberg}. For each $s$ in $S$, define $\theta_s$ to be the partial bijection on $S - \{0\}$ with domain 
\[
	F_s = \{ s \in S : sx \neq 0 \},
\]
and range
\[
	E_s = \{y \in S : y = sx \neq 0 \text{ for some } x \in S \}.
\]
Then the \emph{inverse hull} of $S$ is defined to be the inverse semigroup generated by $\{ \theta_s : s \in S \text{ and } s \neq 0 \}$.
 
\section{The semilattice of the inverse hull of a Markov shift}

In this section we consider the inverse hull of a Markov shift. Our goals are to develop a thorough understanding of the semilattice and to find properties that characterize these inverse semigroups. 

Let $A$ be a finite alphabet and let $T = \{ T_{a,b} \}_{a,b \in A}$ be a matrix such that $T_{a,b} \in \{0,1\}$ for each $a,b \in A$. We refer to $T$ as a \emph{Markov transition matrix}. The set of infinite words $a_1 a_2 a_3 \dots$ such that 
\[
	T_{a_i, a_{i+1}} = 1 \text{ for all $i$ in $\mathbb{N}$}
\]
is called the Markov subshift associated with the transition matrix $T$. Let $L_T$ consist of all finite words occurring as subwords of elements of the Markov subshift of $T$. Notice that if there is a row in $T$ consisting only of zeros, then no word in $L_T$ will contain the letter associated with that row. Therefore we will always assume that $T$ contains no zero rows. 

Let $S_T = L_T \cup \{0\}$. For $x,y$ in $L_T$ define
\[ x*y = \left\{\begin{array}{ll}
        xy & \mbox{if $xy \in L_T$} \\
        0 & \mbox{otherwise}
   \end{array} \right. \]
where $xy$ denotes the concatenation of words $x$ and $y$. Then $S_T$ is a $0$-cancellative semigroup under this operation. Note that $S_T$ does not contain a multiplicative identity. It will sometimes be useful to adjoin one. We denote by $S_T^{1}$ the set $S_T$ with an identity adjoined. Similarly, $L_T^{1} = L_T \cup \{1\}$.

Recall that $H(S_T)$ is the inverse semigroup with $0$ generated by the maps $\theta_w$ such that $w \in L_T$. It will be useful to develop some basic properties about products of such maps.

\begin{lem}\label{lem:hullproducts} For $u,v$ in $L_T$ we have 
\[ \theta_u \theta_v = \begin{cases}
		\theta_{uv} & \text{ if } uv \in L_T, \\
		0 & \text{ otherwise. }
	\end{cases}
\]
	
Also,
\[ \theta_u^{-1} \theta_v = \begin{cases}
		\theta_{v'} & \text{ if } v = uv' \text{ for some } v' \in L_T, \\
		\theta_{u'}^{-1} & \text{ if } u = vu' \text{ for some } u' \in L_T, \\
		\theta_u^{-1} \theta_u & \text{ if } u = v \text{, and} \\
		0 & \text{ otherwise. }
		\end{cases}
\]

\end{lem}
\begin{proof} We leave the straightforward proof of the first assertion to the reader. For the second, suppose first that $v = uv'$. It follows that
\[
\dom \theta_u^{-1} \theta_v = \dom \theta_v = \dom \theta_{v'},
\]
and that for any $x \in \dom \theta_{v'}$ we have $\theta_u^{-1} \theta_v(x) = v'x = \theta_{v'}(x)$. Thus $\theta_u^{-1} \theta_v = \theta_{v'}$. The case where $u = vu'$ for some $u'$ in $L_T$ is similar, and the case where $u = v$ is vacuously true. Finally, suppose that there are no $x,y \in L_T^{1}$ such that $vx = uy$. Note that if a word $x$ lies in $\dom \theta_u^{-1} \theta_v$ then there exists some $y \in L_T$ such that $vx = uy$. Thus $\dom(\theta_u^{-1} \theta_v) = \emptyset$ and $\theta_u^{-1} \theta_v = 0$.
\end{proof}

\begin{rem}\label{rem:cancellation} One useful consequence of Lemma \ref{lem:hullproducts} is that for $u,v \in L_T$, 
\begin{equation*}
	\theta_u^{-1} \theta_u \theta_v = \begin{cases}
		\theta_v & \text{ if } uv \in L_T \\
	 	0 & \text{ otherwise }
	 \end{cases}
\end{equation*}
\end{rem}

We can also use the lemma to find a general form for elements in $H(S_T)$. We omit most of the proof of the following theorem, since similar results have appeared in both \cite{StarlingShifts} and \cite{ExelSteinberg}. We point out that these forms are not unique.

\begin{thm} All nonzero elements of $H(S_T)$ are of the form 
\[
\theta_s\theta^{-1}_{x_1}\theta_{x_1}\dots\theta^{-1}_{x_n}\theta_{x_n}\theta^{-1}_w,
\]
for some $n \geq 1,$ $x_i \in A,$ and $s,w \in L_T^{1}$.
\end{thm}
\begin{proof}
Note that this form is nearly identical to the one given in \cite[Theorem 7.11]{ExelSteinberg}. There are two small differences to consider. The first is that the idempotents in the middle of the product are associated with the letters of $A$, rather than arbitrary words. We do not lose generality, since for $w \in L_T$, $\theta_w^{-1} \theta_w = \theta_a^{-1} \theta_a$ where $a$ is the last letter of $w$. Second, we do not assume that the product of idempotents in the middle includes $\theta_s^{-1} \theta_s$ or $\theta_w^{-1} \theta_w$. Of course, this is not consequential since 
\[
\theta_s\theta^{-1}_{x_1}\theta_{x_1}\dots\theta^{-1}_{x_n}\theta_{x_n}\theta^{-1}_w = \theta_s (\theta_s^{-1} \theta_s) \theta^{-1}_{x_1}\theta_{x_1}\dots\theta^{-1}_{x_n}\theta_{x_n}(\theta_w^{-1} \theta_w)\theta^{-1}_w.
\]
\end{proof}

It follows quickly that the nonzero idempotents of $H(S_T)$ are of the form 
\[
\theta_s\theta^{-1}_{x_1}\theta_{x_1}\dots\theta^{-1}_{x_n}\theta_{x_n}\theta^{-1}_s
\]
where $n \geq 1,$ $x_i \in A,$ and $s \in L_T^{1}$. Based on this, it is natural to distinguish the idempotents for which $s=1$ from those for which $s \neq 1$. As we will see, the first group of idempotents sits above the second in the semilattice. In between the two groups is the set $\mcO = \{ \theta_a \theta_a^{-1} : a \in A\}$ which are the range idempotents associated with letters. 

Note that $\theta_a \theta_a^{-1}$ is the identity map on the set 
\[
aL_T := \{u \in L_T : u = av \text{ for some } v \in L_T \}.
\]  
Choose $a,b \in A$ with $a \neq b$. Since $aL_T \cap bL_T = \emptyset$, we have $\theta_a \theta_a^{-1} \theta_b \theta_b^{-1} = 0$. Thus we say that $\mcO$ consists of \emph{mutually orthogonal} idempotents. Now define $\mcO^{\up} = \{\alpha \in E(H(S_T)) : \alpha \geq \theta_a \theta_a^{-1} \text{ for some } a \in A \}$ and $\mcO^{\down} = \{\alpha \in E(H(S_T)) : \alpha \leq \theta_a \theta_a^{-1} \text{ for some } a \in A \}$. 

\begin{prop}\label{prop:idempotents} Let $H(S_T)$ be the inverse hull of the semigroup $S_T$ associated with a Markov transition matrix $T$. Set $\mcO = \{ \theta_a \theta_a^{-1} : a \in A\}$. Then
$\mcO^{\up} - \mcO = \{ \theta^{-1}_{x_1}\theta_{x_1}\dots\theta^{-1}_{x_n}\theta_{x_n} : x_i \in A\} - \{0\}$ and
\[
\mcO^{\down} = \{ \theta_s\theta^{-1}_{x_1}\theta_{x_1}\dots\theta^{-1}_{x_n}\theta_{x_n}\theta^{-1}_s : x_i \in A, s \neq 1 \} \cup \{0\}.
\]

\end{prop}
\begin{proof}
We start with the set $\mcO^{\up} - \mcO$. As distinct idempotents of $\mcO$ are incomparable, note that 
\[
	\mcO^{\up} - \mcO = \{\tau \in E(H(S_T)) : \tau > \theta_a \theta_a^{-1} \text{ for some } a \in A \}.
\]

Let $\tau = \theta^{-1}_{x_1}\theta_{x_1}\dots\theta^{-1}_{x_n}\theta_{x_n}$ be nonzero, where $x_i \in A$ for $1 \leq i \leq n$. Since $\tau \neq 0$, there is a letter $a$ in $A$ such that $T_{x_i, a} = 1$ for all $1 \leq i \leq n$. So $x_i a \in L_T$ for each $i$. It follows that
\[
 \dom(\theta_a \theta_a^{-1}) = aL_T \subset \{ w \in L_T : x_i w \in L_T \text{ for all } i \} = \dom \tau.
\]
The above inclusion is proper because $a \in \dom \tau$ while $aL_T$ does not contain letters. Thus $\tau \in \mcO^{\up} - \mcO$.  

Now suppose that $\tau > \theta_a \theta_a^{-1}$ for some $a \in L_T$. Then $\tau \neq 0$ and we can write $\tau = \theta_s\theta^{-1}_{x_1}\theta_{x_1}\dots\theta^{-1}_{x_n}\theta_{x_n}\theta^{-1}_s$ where $x_i \in A$ and $s \in L_T^{1}$. Suppose that $s \neq 1$. Then $\theta_s \theta_s^{-1} > \theta_a \theta_a^{-1}$. It follows that $aL_T \subseteq sL_T$. By our assumption that the row corresponding to $a$ in $T$ is nonzero, we know that $a L_T \neq \emptyset$. So there is some $b \in A$ such that $ab = sw$ for some word $w \in L_T$ and thus $a = sw'$ for some $w' \in L_T^{1}$. Thus $s = a$ and $w' = 1$. But this contradicts the inequality $\theta_s \theta_s^{-1} > \theta_a \theta_a^{-1}$. Therefore $s = 1$ and $0 \neq \tau = \theta^{-1}_{x_1}\theta_{x_1}\dots\theta^{-1}_{x_n}\theta_{x_n}$.

Next we consider the set $\mcO^{\down}$. Let $\tau = \theta_s\theta^{-1}_{x_1}\theta_{x_1}\dots\theta^{-1}_{x_n}\theta_{x_n}\theta^{-1}_s$ where $s \neq 1$. Note that $\dom \alpha \subseteq sL \subseteq \dom(\theta_{l(s)} \theta_{l(s)}^{-1})$ where $l(s)$ denotes the last letter of $s$. Thus $\tau \leq \theta_{l(s)} \theta_{l(s)}^{-1}$. Conversely, suppose that $0 \neq \tau \leq \theta_a \theta_a^{-1}$ for some $a \in A$ and write $\tau = \theta_s\theta^{-1}_{x_1}\theta_{x_1}\dots\theta^{-1}_{x_n}\theta_{x_n}\theta^{-1}_s$. If $s \neq 1$ we are done, so suppose $s=1$. Then we have $\theta_{x_i}^{-1} \theta_{x_i} \theta_a \theta_a^{-1} \neq 0$ for all $i$. But one can quickly check that this is equivalent to $\theta_{x_i}^{-1} \theta_{x_i} \theta_a \theta_a^{-1} = \theta_a \theta_a^{-1}$. Therefore $\tau = \tau \, \theta_a \theta_a^{-1} = \theta_a \theta_a^{-1} = \theta_a (\theta_a^{-1} \theta_a) \theta_a^{-1}$. Since $a \neq 1$, this is of the required form.

\end{proof}

At this point we collect some observations that follow from the above proposition. These properties will feature in our characterization of the inverse hulls of Markov subshifts. The first observation is that each idempotent of $H(S_T)$ is comparable to some element in $\mcO$. The second observation is that $(\mcO^{\up} - \mcO) \cup \{0\}$ is closed under multiplication. Finally, one can quickly show that $\mcO^{\up} \cup \{0\}$ is also closed under multiplication. Elements of $\mcO$ are mutually orthogonal, which tells us that products of distinct elements of $\mcO$ are zero. Also, by Remark \ref{rem:cancellation}, products of the form $\theta^{-1}_{x_1}\theta_{x_1}\dots\theta^{-1}_{x_n}\theta_{x_n} ( \theta_a \theta_a^{-1})$ are either $0$ or $\theta_a \theta_a^{-1}$.

Now we will develop some properties for our characterization. The next Proposition says that an element of $\mcO^{\up} - \mcO$ is uniquely determined by the idempotents in $\mcO$ that it sits above.

\begin{prop} Fix $\alpha, \beta$ in $\mcO^{\up} - \mcO$. Suppose that $\mcO \cap \alpha^{\down} = \mcO \cap \beta^{\down} $.
Then $\alpha = \beta$.
\end{prop}
\begin{proof} Write $\alpha = \theta^{-1}_{x_1}\theta_{x_1}\dots\theta^{-1}_{x_n}\theta_{x_n}$. Note that $\theta_a \theta_a^{-1} \in \mcO \cap \alpha^{\down}$ if and only if $\theta_{x_i}^{-1}\theta_{x_i}\theta_a \theta_a^{-1} \neq 0$ for $1 \leq i \leq n$, which is equivalent to $x_i a \in L_T$ for $1 \leq i \leq n$. Also
\begin{align*}
 \dom \alpha &= \{w \in L_T : x_i w \in L_T \text{ for } 1 \leq i \leq n \} \\
 		&= \bigcup \left\{ aL_T^{1} : a \in A, x_i a \in L_T \text{ for all } i \right\}.
\end{align*}
Therefore, if $\mcO \cap \alpha^{\down} = \mcO \cap \beta^{\down}$, then $\dom \alpha = \dom \beta$. Since $\alpha$ and $\beta$ are idempotents, we conclude that $\alpha = \beta$.
\end{proof}

Finally, we consider some properties of the $\mcD$-classes and $\mcH$-classes of $H(S_T)$.

\begin{prop}\label{prop:Dclass} Fix $\alpha, \beta$ in $\mcO^{\up} - \mcO$. If $\alpha \mcD \beta$ then $\alpha = \beta$. Also, for each $\theta_a \theta_a^{-1}$ in $\mcO$, there exists a unique $\alpha$ in $\mcO^{\up} - \mcO$ such that $\theta_a \theta_a^{-1} \mcD \alpha$.
\end{prop}
\begin{proof} Consider arbitrary $\gamma = \theta_s\theta^{-1}_{x_1}\theta_{x_1}\dots\theta^{-1}_{x_n}\theta_{x_n}\theta^{-1}_w$ in $H(S_T)$ and suppose that $\gamma \gamma^*$ and $\gamma^* \gamma$ lie in $\mcO^{\up} - \mcO$. By the same argument as in the proof of Proposition \ref{prop:idempotents}, $\gamma \gamma^* \in \mcO^{\up} - \mcO$ implies $s = 1$. Also, $\gamma^* \gamma \in \mcO^{\up} - \mcO$ implies $w = 1$. Therefore $\gamma \gamma^* = \gamma^* \gamma$. Now suppose that $\alpha, \beta$ in $\mcO^{\up} - \mcO$ with $\alpha \mcD \beta$. Then there exists $\gamma \in H(S_T)$ such that $\gamma \gamma^* = \alpha$ and $\gamma^* \gamma = \beta$. By the above argument, $\alpha = \beta$.

For the second assertion, we have that $\theta_a \theta_a^{-1} \mcD \theta_a^{-1} \theta_a$. Since no row of the Markov transition matrix is zero, we have $0 \neq \theta_a^{-1} \theta_a \in \mcO^{\up} - \mcO$, which completes the proof.
\end{proof}

\begin{prop} The inverse hull $H(S_T)$ is combinatorial.
\end{prop}

\begin{proof} Let $\gamma = \theta_s\theta^{-1}_{x_1}\theta_{x_1}\dots\theta^{-1}_{x_n}\theta_{x_n}\theta^{-1}_w \neq 0$ and suppose that $\gamma \gamma^* = \gamma^* \gamma$. Now, $\dom \gamma^* \gamma = \{ wu \in L_T : u \in L_T, su \in L_T, \text{ and } x_i u \in L_T \text{ for all $i$} \}$ and $\dom \gamma \gamma^* = \{ su \in L_T : u \in L_T, wu \in L_T, \text{ and } x_i u \in L_T \text{ for all $i$} \}$.

Since $\gamma \gamma^* = \gamma^* \gamma \neq 0$ we have $wu = sv$ for some $u,v \in L_T$. Then $w = sy$ or $s = wy$ for some $y \in L_T^{1}$. First suppose that $w = sy$. As $\gamma^* \gamma \neq 0$, there exists $a \in A$ such that $wa \in \dom \gamma^* \gamma$. So $sa \in L_T$ and $x_i a \in L_T$ for all $i$. Then $sa \in \dom \gamma \gamma^*$ and hence $sa \in \dom \gamma^* \gamma$. Then we have that $sa = wz$ for some $z \in L_T$. By removing the last letter we have $s = wz'$ for some $z' \in L_T^{1}$. Since $w = sy$ for some $y \in L_T^{1}$ and $s = wz'$ for some $z' \in L_T^{1}$, we conclude that $s = w$. A symmetric argument works in the case that $s = wy$. 

Therefore $\gamma$ is idempotent, which completes the proof that $H(S_T)$ is combinatorial.

\end{proof}

\section{The Characterization of Inverse Hulls of Markov Shifts}

In this section we show that an inverse semigroup $H$ is isomorphic to the inverse hull of a Markov shift if and only if it is combinatorial and it contains a set $\mcO$ of mutually orthogonal nonzero idempotents satisfying certain properties. Fix a combinatorial inverse semigroup $H$ with $0$ and a set $\mcO$ of nonzero idempotents such that $\mcO$ satisfies the following properties:

\begin{enumerate}

\item[(O1)] the elements of $\mcO$ are mutually orthogonal,

\item[(O2)] every idempotent in $H$ is comparable to some element of $\mcO$,

\item[(O3)] both $\mcO^{\up} \cup \{0\}$ and $(\mcO^{\up} - \mcO) \cup \{0\}$ are closed under multiplication,

\item[(O4)] elements of $\mcO^{\up} - \mcO$ are uniquely determined by the set of idempotents in $\mcO$ that they lie above in the natural partial order, and

\item[(O5)] for each $e \in \mcO$, the $\mcD$-class of $e$ contains at most one element of $\mcO^{\up} - \mcO$.

\end{enumerate}

Using the above assumptions, we will show that $H$ contains a $0$-left cancellative semigroup $S$ isomorphic to the semigroup associated with a Markov transition matrix $T$. When $S$ generates $H$, we show that $H$ is isomorphic to $H(S_T)$. The first order of business it to define the sets that will serve as our alphabet and language respectively. Let 
\begin{align*}
 A &=  \{ a \in H : a^*a \in \mcO^{\up} - \mcO \text{ and } aa^* \in \mcO \} \text{, and } \\
 L &= \{ a_1 a_2 \dots a_n \neq 0 : n \in \N, a_i \in A \}.
\end{align*}

We refer to $L$ as the \emph{language associated with $\mcO$}, a name that is justified by Theorem \ref{thm:language} below. We show that $L$ behaves under multiplication much like the set of generators of the inverse hull of a Markov shift. In particular, compare Corollary \ref{cor:generalproducts} below with Lemma \ref{lem:hullproducts}.

\begin{prop}\label{prop:letterproducts} Let $H$ be a combinatorial inverse semigroup and suppose $\mcO$ is a set of nonzero idempotents that satisfies conditions (O1) -- (O5). For $a,b \in A$ we have:
\begin{enumerate}

\item $ab \neq 0$ if and only if $bb^* \leq a^*a$, and
\item $a^*b \neq 0$ if and only if $a = b$.

\end{enumerate} 
\end{prop}
\begin{proof} 
Note that both $a$ and $b$ are nonzero. Suppose $bb^* \leq a^*a$. Then $a^*abb^* = bb^* \neq 0$. Thus $ab \neq 0$. 

Conversely, suppose $ab \neq 0$. Then $a^*abb^* \neq 0$. Since $\mcO^{\up} \cup \{0\}$ is closed under multiplication by (O3), $a^*abb^* \in \mcO^{\up}$. That is, $a^*abb^* \geq cc^*$ for some $c \in A$. Also $a^*abb^* \leq bb^*$. Since elements of $\mcO$ are incomparable, $a^*abb^* = bb^*$. So $a^*a \geq bb^*$.

Next suppose $a = b$. Then $aa^*b = a$ implies that $a^*b \neq 0$. 

Conversely, if $a^*b \neq 0$ then $aa^*bb^* \neq 0$. By (O1), $aa^* = bb^*$. Then for $x = a^*b$ we have $xx^* = a^*a$ and $x^*x = b^*b$. Thus $aa^* \mcD a^*a \mcD b^*b$. By (O5), $a^*a = b^*b$. Since $H$ is combinatorial, $a = b$.
\end{proof}

As a corollary we have the following property about products involving words in $L$.

\begin{cor}\label{cor:generalproducts} Let $u,v \in L$ and $u = xa$ for some $x$ in $L^{1}$ and $a$ in $A$. Then
\[
	u^*v = \begin{cases}
		w	& \text{if } v = uw \text{ for some } w \in L, \\ 
		w^*	& \text{if } u = vw \text{ for some } w \in L, \\
		a^*a & \text{if } u = v, \text{ and } \\
		0 & \text{otherwise. }
	\end{cases}
\]
\end{cor}
\begin{proof} Let $u = a_1 a_2 \cdots a_n$ and $v = b_1 b_2 \cdots b_m$ where $a_i, b_j \in L$. If $u^*v \neq 0$ then it follows from Proposition \ref{prop:letterproducts} that $a_1 = b_1$. Moreover since $b_1 b_2 \neq 0$ we have $b_2 b_2^* \leq b_1^* b_1$. Then
\begin{align*}
	u^*v	&= a_n^* \cdots a_2^* a_1^* b_1 b_2 \cdots b_m \\
			&= a_n^* \cdots a_2^* (b_1^* b_1) (b_2 b_2^*) b_2 \cdots b_m \\
			&= a_n^* \cdots a_2^* b_2 \cdots b_m.
\end{align*}
By continuing in this way we conclude that $u = v$, $u = vw$, or $v = uw$ for some $w \in L$. In each case, the formula for $u^*v$ can be verified quickly using similar calculations.
\end{proof}

Just as for $H(S_T)$, we can now show that products in the inverse semigroup generated by $L$ take on a special form. We leave the proof of the following proposition to the reader.

\begin{prop}\label{prop:generalnormalform} Any nonzero element of the inverse semigroup generated by $L$ is of the form $w a_1^* a_1 a_2^* a_2 \cdots a_n^* a_n v^*$ where $n \geq 1$, $a_i \in A$ and $w,v \in L^{1}$.
\end{prop}

Since $n \geq 1$ the product $w a_1^* a_1 a_2^* a_2 \cdots a_n^* a_n v^*$, which is formally defined in $H^{1}$, is always an element of $H$. It follows from Proposition \ref{prop:generalnormalform} that the idempotents of the inverse semigroup generated by $L$ are of the form $w a_1^* a_1 a_2^* a_2 \cdots a_n^* a_n w^*$ where $n \geq 1$, $a_i \in A$ and $w \in L^{1}$.

Another useful consequence of the above proposition is that when $L$ generates $H$ (as an inverse semigroup with $0$) the sets $\mcO$ and $\mcO^{\up} - \mcO$ have a form that echoes their counterparts in the inverse hull of a Markov subshift.

\begin{prop} Suppose that $H$ is generated by $L$. Then $\mcO = \{ aa^* : a \in A \}$ and $\mcO^{\up} - \mcO = \{ a_1^*a_1 a_2^* a_2 \cdots a_n^* a_n : a_i \in A \} - \{0\}$.

\end{prop}
\begin{proof} By definition of $A$, $aa^* \in \mcO$ for each $a \in A$. Let $e$ be a nonzero idempotent in $\mcO$. We will show that $e = aa^*$ for some $a \in L$. Since $L$ generates $H$, $e = w a_1^* a_1 a_2^* a_2 \cdots a_n^* a_n w^*$ where $n \geq 1$, $a_i \in A$ and $w \in L^{1}$. First suppose that $w = 1$. By definition of $A$, $a_i^* a_i \in \mcO^{\up} - \mcO$ for each $i$. By (O3), $0 \neq e = a_1^* a_1 a_2^* a_2 \cdots a_n^* a_n \in \mcO^{\up} - \mcO$, a contradiction. Thus $w \neq 1$ and we may write $w = a w'$ for some $a \in A$ and $w' \in L^{1}$. Then $e \leq aa^* \in \mcO$. As idempotents in $\mcO$ are mutually orthogonal, we conclude that $e = aa^*$.

We have shown in the previous paragraph that if $a_1^* a_1 a_2^* a_2 \cdots a_n^* a_n \neq 0$ then $a_1^* a_1 a_2^* a_2 \cdots a_n^* a_n \in \mcO^{\up} - \mcO$. Let $e \in \mcO^{\up} - \mcO$. Write
\[
e = w a_1^* a_1 a_2^* a_2 \cdots a_n^* a_n w^*
\]
where $n \geq 1$, $a_i \in A$ and $w \in L^{1}$. If $w \neq 1$ then write $w = a w'$ where $a \in A$ and $w' \in L^{1}$ and note that $e \leq aa^*$, contradicting our assumption that $e \in \mcO^{\up} - \mcO$. Thus $e = a_1^* a_1 a_2^* a_2 \cdots a_n^* a_n$.
\end{proof}

\subsection{$L^0$ is the $0$-left cancellative semigroup of a Markov subshift}

Here we continue to assume that $H$ is a combinatorial inverse semigroup and that $\mcO$ is a set of nonzero idempotents satisfying (O1) -- (O5). We also retain the notation $A$ and $L$. The set $A^*$ denotes finite words over the alphabet $A$. If $a_i \in A$ for $1 \leq i \leq n$, we will temporarily write $a_1 \circ a_2 \circ a_3 \circ \cdots \circ a_n$ for a word in $A^*$ since $a_1 a_2 \cdots a_n$ represents a product in $H$. Let 
\[
M = \{ a_1 \circ a_2 \circ a_3 \circ \cdots \circ a_n \in A^* : a_1 a_2 \cdots a_n \in L \}. 
\]

\begin{prop} The map $a_1 \circ a_2 \circ a_3 \circ \cdots \circ a_n  \mapsto a_1 a_2 \cdots a_n$ from $M$ to $L$ is a bijection.
\end{prop}
\begin{proof}
The map is clearly surjective. Suppose that $w = a_1 a_2 \cdots a_n = b_1 b_2 \cdots b_m$ in $L$ where $a_i, b_j \in A$. Then $0 \neq w^*w = a_n^* \cdots a_2^* a_1^* b_1 b_2 \cdots b_m$. By the same argument used in the proof of Corollary \ref{cor:generalproducts}, we conclude that $a_1 = b_1$ and $w^*w = a_n^* \cdots a_2^* b_2 \cdots b_m$. We may continue in this way until we exhaust all the $a_i$ or $b_j$. Notice by Corollary \ref{cor:generalproducts} we also have that $w^*w = a_n^* a_n$. If $n<m$ we have $a_n^* \cdots a_{n-m+1}^* = w^* w = a_n^* a_n \in \mcO^{\up} - \mcO$. Then there exists $a \in A$ such that $a_n^* \cdots a_{n-m+1}^* (aa^*) = aa^*$. This implies that $a = a_{n-m+1}$ and $a_n^* \cdots a_{n-m+1}^* = a_n^* \cdots a_{n-m+1}^* (aa^*) = aa^* \in \mcO$, which contradicts the fact that $a_n^* \cdots a_{n-m+1}^* \in \mcO^{\up} - \mcO$. If $m < n$, then $b_{m-n+1} \cdots b_n = w^* w \in \mcO^{\up} - \mcO$. Similarly, there exists $bb^* \in \mcO$ such that $(bb^*)b_{m-n+1} \cdots b_n = bb^*$. Again this leads to the contradiction that $b_{m-n+1} \cdots b_n \in \mcO$. Thus $n = m$ and $a_i = b_i$ for $1 \leq i \leq n$, which shows that the map under consideration is injective.
\end{proof}

The proposition shows that we may identify $L$ with a collection $M$ of words in $A^*$. We do so for the rest of this section. Under this identification, we show that $L$ is in fact the language of a Markov subshift. Define an $A \times A$ Markov transition matrix $T$ by 
\[
	T(a,b) = \begin{cases}
				1 & \text{ if } ab \in L \text{ (i.e. $bb^* \leq a^*a$)}\\
				0 & \text{ otherwise. }
			\end{cases}
\]

\begin{thm}\label{thm:language} The semigroup $L \cup \{0\}$ is isomorphic to the semigroup $S_T = L_T \cup \{0\}$ associated with the Markov transition matrix $T$.
\end{thm}
\begin{proof} We just need to verify that the set of words in $L$ is equal to $L_T$. Note that for any letter $a \in A$, $a^*a \in \mcO^{\up} - \mcO$ and thus there is a letter $b \in A$ such that $bb^* \leq a^*a$. So $ab \in L$. Then any $w = a_1 a_2 \cdots a_n \in L$ can be extended to an infinite word $a_1 a_2 \cdots a_n a_{n+1} a_{n+2} \cdots$ such that $T(a_i, a_{i+1}) = 1$ for each $i \in \N$. Thus $w \in L_T$. Conversely, let $w$ be a subword of an infinite word in the subshift associated with $T$. Then we may write $w = a_1 a_2 \cdots a_n$ where $T(a_i, a_{i+1}) = 1$ for each $1 \leq i \leq n-1$. It follows that $a_i a_{i+1}$ in $L$ for each $1 \leq i \leq n-1$. Note that the proof of the $u=v$ case in Corollary \ref{cor:generalproducts} depends only on the assumption that products of consecutive letters of $u$ lie in $L$. Thus if $w = 0$, then $a_n^* a_n = w^* w = 0$. By contradiction, $w \neq 0$ and so $w \in L$.

\end{proof}

It follows that $L\cup \{0\}$ is a $0$-cancellative subsemigroup of $H$.

\subsection{The isomorphism}
Here we show that if $L$ generates $H$ as an inverse semigroup with zero, then $H$ is isomorphic to $H(L^{0})$. As a consequence, an inverse semigroup $H$ is isomorphic to the inverse hull of a Markov subshift if and only if $H$ is combinatorial, $H$ contains a set of nonzero idempotents $\mcO$ satisfying (O1) -- (O5), and $L$ generates $H$.

We continue to fix a combinatorial inverse semigroup $H$ that contains a set of nonzero idempotents $\mcO$ satisfying (O1) -- (O5). Given $\alpha \in H$, define
\[
	D_{\alpha} = \{ x \in L : xx^* < \alpha^* \alpha \}.
\]
We say that $H$ is \emph{right reductive relative to $L$} if for all $\alpha, \beta \in H$,
\begin{enumerate}
\item $D_{\alpha} = D_{\beta}$ and
\item $\alpha x  = \beta x$ for all $x \in D_{\alpha}$
\end{enumerate}
implies $\alpha = \beta$.

\begin{lem}\label{lem:domainsets} Let $0 \neq \alpha = w a_1^* a_1 a_2^* a_2 \cdots a_n^* a_n v^*$ where $n \geq 1$, $a_i \in A$, $w,v\in L^{1}$. Then:
\begin{enumerate}

\item For $v \neq 1$: $x \in D_{\alpha}$ if and only if $x = vy$ for some $y \in L$ and $\alpha x = wy \in L$.

\item For $v = 1$: $x \in D_{\alpha}$ if and only if $wx \in L$ and $a_i x \in L$ for each $i$.

\end{enumerate}

\end{lem}
\begin{proof}
We give the proof for $v \neq 1$. The case where $v=1$ is similar. Let $x \in D_{\alpha}$. Then $xx^* < \alpha^* \alpha \leq vv^*$. Hence $v^*x \neq 0$. We can then employ Corollary \ref{cor:generalproducts} to conclude that $x = vy$ for some $y \in L$. Moreover, as $vy \neq 0$, we have that $v^*vy = y$ by Corollary \ref{cor:generalproducts}. By similar reasoning we have 
\[
\alpha x = w a_1^* a_1 \cdots a_n^* a_n (v^* v y) = w a_1^* a_1 \cdots (a_n^* a_n y) = \dots = wy.
\]
Since $x \in D_{\alpha}$ we have $\alpha x \neq 0$ and hence $\alpha x = wy \in L$.

Conversely suppose $x = vy \in L$ for some $y \in L$ and $\alpha x = wy \in L$. Then $\alpha x \neq 0$ and hence
\begin{align*}
	0 \neq \alpha^* \alpha xx^*	&= v a_1^* a_1 \cdots a_n^* a_n w^*w v^* v x^* \\
								&= v a_1^* a_1 \cdots a_n^* a_n (w^*w y) x^* \\
								&= v a_1^* a_1 \cdots (a_n^* a_n y) x^* \\
								&= vyx^* \\
								&= xx^*.
\end{align*}

So $xx^* \leq \alpha^* \alpha$. Suppose $xx^* = \alpha^* \alpha$. As $vy \in L$ we have $yy^* \leq v^*v$ so
\[
	0 \neq yy^* = v^*v yy^* v^*v = v^* xx^* v = v^*v a_1^* a_1 \cdots a_n^* a_n w^* w.
\]
Thus $yy^* \in \mcO^{\up} - \mcO,$ a contradiction. Therefore $xx^* < \alpha^* \alpha$.

\end{proof}
\begin{prop} If $L$ generates $H$, then $H$ is right reductive relative to $L$. 
\end{prop}
\begin{proof}
First we show that $\alpha \neq 0$ implies $D_{\alpha} \neq \emptyset$. If $\alpha \neq 0$ then $\alpha = w a_1^* a_1 a_2^* a_2 \cdots a_n^* a_n v^*$ for some $n \geq 1$, $a_i \in A$ and $w,v \in L^{1}$. As $0 \neq \alpha^* \alpha = v a_1^*a_1 \cdots a_n^*a_n w^*w v^*$ we have $v^*v a_1^*a_1 \cdots a_n^*a_n w^*w \neq 0$. Thus 
\[
v^*v a_1^*a_1 \cdots a_n^*a_n w^*w \in \mcO^{\up} - \mcO
\]
and there exists $a \in A$ such that $aa^* < v^*v a_1^*a_1 \cdots a_n^*a_n w^*w = v^* \alpha^* \alpha v$. Then $\alpha v a \neq 0$, since $(\alpha v a)^* \alpha v a = a^*a$. For $v = 1$ we have $\alpha a \neq 0$ implies that $wa \neq 0$ and $a_ia \neq 0$ for all $i$. Thus by (2) of Lemma \ref{lem:domainsets}, $a \in D_{\alpha}$. For $v \neq 1$ we have
\[
	\alpha va = w a_1^*a_1 \cdots a_n^* a_n (v^* v a) = w a_1^*a_1 \cdots (a_n^* a_n a) = \dots = wa.
\]
Thus $va \in D_{\alpha}$ by Lemma \ref{lem:domainsets}.

Now for $\alpha, \beta$ in $H$ suppose that $D_{\alpha} = D_{\beta}$ and $\alpha x = \beta x$ for all $x \in D_{\alpha}$. By the argument above, if $D_{\alpha} = D_{\beta} = \emptyset$, then $\alpha = \beta = 0$. So we may assume $D_{\alpha} \neq \emptyset$. Write $\alpha = w a_1^* a_1 a_2^* a_2 \cdots a_n^* a_n v^*, \beta = t b_1^* b_1 b_2^* b_2 \cdots b_m^* b_m z^*$ where $m,n \geq 1$, $a_i, b_j \in A$, and $w,v,t,z \in L^{1}$. Thus $vy \in D_{\alpha}$ for some $y \in L$. In fact for any prefix $y_1$ of $y$ we have $vy_1 \in D_{\alpha}$. 
In particular, $va \in D_{\alpha}$ for some $a \in A$. Then there exists $x \in L$ such that $va = zx$. By right cancellativity we conclude that $v = zx'$ for some $x' \in L^{1}$. Similarly $z = vy'$ for some $y' \in L^{1}$. Then $v = vy'x'$ implies that $y' = x' = 1$ and we conclude that $v = z$. It follows that
\[
wa = \alpha va = \alpha za = \beta z a = ta
\] 
and hence $w = t$.

To finish the proof we will show that 
\[
	v^* \alpha^* \alpha v = w^*w a_1^*a_1 \cdots a_n^*a_n v^*v = t^*t b_1^*b_1 \cdots b_n^*b_n z^*z = z^* \beta^* \beta z.
\]
By (O4), it suffices to prove that $(v^* \alpha^* \alpha v)^{\down} \cap \mcO = (z^* \beta^* \beta z)^{\down} \cap \mcO$. To that end, suppose $bb^* \in (v^* \alpha^* \alpha v)^{\down} \cap \mcO$. Then $0 \neq vbb^*v^* \leq \alpha^* \alpha$. If $vbb^*v^* = \alpha^* \alpha$ then $bb^* = v^* \alpha^* \alpha v \in \mcO^{\up} - \mcO$, a contradiction. Thus $zb = vb \in D_{\alpha} = D_{\beta}$. It follows that $bb^* \in (z^* \beta^* \beta z)^{\down} \cap \mcO$. Thus $(v^* \alpha^* \alpha v)^{\down} \cap \mcO \subseteq (z^* \beta^* \beta z)^{\down} \cap \mcO$. The other inclusion follows by a symmetric argument.

By (O4) we have that $v^* \alpha^* \alpha v = z^* \beta^* \beta z$ and thus 
\[
 \alpha = w v^* \alpha^* \alpha v v^* = t z^* \alpha^* \alpha z z^* = \beta.
\]
\end{proof}

We now show that the set $D_\alpha$ is the domain of the map in $H(L^{0})$ naturally associated with $\alpha$.

\begin{prop}\label{prop:domain} Let $\alpha = w a_1^* a_1 \cdots a_n^* a_n v^*$ in $H$ and consider 
\[
\varphi = \theta_w \theta^{-1}_{a_1}\theta_{a_1}\cdots\theta^{-1}_{a_n}\theta_{a_n}\theta^{-1}_v
\]
in $H(L^0)$. Then $D_{\alpha} = \dom \varphi$.
\end{prop}
\begin{proof}
Note that $x \in \dom \varphi$ if and only if $y \in L$ with $x = vy$ for some $y \in L$ for which $a_i y \neq 0$ for $1 \leq i \leq n$ and $wy \neq 0$. By Lemma \ref{lem:domainsets}, $D_{\alpha} = \dom \varphi$.
\end{proof}

We can now prove the main theorem of the section which characterizes the inverse semigroups that are isomorphic to inverse hulls of Markov subshifts.

\begin{thm}\label{thm:main} Let $H$ be an inverse semigroup with $0$. Then $H$ is isomorphic to the inverse hull of a Markov subshift if and only if
\begin{enumerate}
\item $H$ is combinatorial,
\item there is a set $\mcO$ of nonzero idempotents in $H$ satisfying $(O1) - (O5)$, and 
\item the language, $L$, associated with $\mcO$ generates $H$.
\end{enumerate}
\end{thm}
\begin{proof}
First, it was proved in section 3 that the inverse hull of a Markov subshift satisfies properties (1), (2), and (3). Next let $H$ satisfy properties (1), (2), and (3). Then $L$ is the language of a Markov subshift by Theorem \ref{thm:language}. We would like to define a map $\Gamma : H \to H(L^{0})$ by sending $0$ to $0$ and $\alpha = w a_1^* a_1 \cdots a_n^* a_n v^*$ to $\theta_w \theta^{-1}_{a_1}\theta_{a_1}\cdots\theta^{-1}_{a_n}\theta_{a_n}\theta^{-1}_v$. First observe that this map is well-defined. Indeed, if $\alpha = w a_1^* a_1 \cdots a_n^* a_n v^* = s b_1^* b_1 \cdots b_n^* b_n t^*$ then 
\[
\dom \theta_w \theta^{-1}_{a_1}\theta_{a_1}\cdots\theta^{-1}_{a_n}\theta_{a_n}\theta^{-1}_v = D_{\alpha} = \dom \theta_s \theta^{-1}_{b_1}\theta_{b_1}\cdots\theta^{-1}_{b_n}\theta_{b_n}\theta^{-1}_t
\]
by Proposition \ref{prop:domain}. Moreover, by Lemma \ref{lem:domainsets}, for each $x \in D_{\alpha},$
\[
\theta_w \theta^{-1}_{a_1}\theta_{a_1}\cdots\theta^{-1}_{a_n}\theta_{a_n}\theta^{-1}_v(x) = wx = \alpha x = sx = \theta_s \theta^{-1}_{b_1}\theta_{b_1}\cdots\theta^{-1}_{b_n}\theta_{b_n}\theta^{-1}_t(x).	
\]
Thus the map $\Gamma$ is well-defined. Suppose that $\Gamma(\alpha) = \Gamma(\beta)$. Then $D_{\alpha} = \dom \Gamma(\alpha) = \dom \Gamma(\beta) = D_{\beta}$ and $\alpha x = \beta x$ for all $x \in D_{\alpha}$. Since $H$ is right reductive on $L$, $\alpha = \beta$. Thus $\Gamma$ is injective. It is clear from the definition of $H(L^{0})$ that $\Gamma$ is surjective.

Finally we need to show that $\Gamma$ is a homomorphism. Let 
\[ 
\alpha = w a_1^* a_1 \cdots a_n^* a_n v^* \text{ and } \beta = s b_1^* b_1 \cdots b_n^* b_n t^*
\] 
be nonzero. As in Corollary \ref{cor:generalproducts}, we have four cases to consider. We examine one such case; the others are similar. Suppose that $v = su$ for some $u \in L$. Then either $\alpha \beta = 0$ or
\begin{align*}
	\alpha \beta &= w a_1^* a_1 \cdots a_n^* a_n (v^* s) b_1^* b_1 \cdots b_n^* b_n t^* \\
				 &= w a_1^* a_1 \cdots a_n^* a_n (u^* b_1^* b_1) \cdots b_n^* b_n t^* \\
				 &= w a_1^* a_1 \cdots a_n^* a_n (u^* b_2^* b_2) \cdots b_n^* b_n t^* \\
				 &\vdots \\
				 &= w a_1^* a_1 \cdots a_n^* a_n (tu)^*
\end{align*}
By Lemma \ref{lem:hullproducts} we the computation of $\Gamma(\alpha) \Gamma(\beta)$ is the same in $H(L^0)$, and so
\[
	\Gamma(\alpha) \Gamma(\beta) = \theta_{w}\theta_{a_1}\cdots\theta^{-1}_{a_n}\theta_{a_n}\theta^{-1}_{tu}
\]
In the other 3 cases described in Corollary \ref{cor:generalproducts}, we can similarly show that $\Gamma$ is a homomorphism. Thus $H \cong H(L^{0})$.
\end{proof}

\section{An application}
In this section we consider whether there are two languages associated with different Markov shifts that generate isomorphic inverse hulls. We find languages $L_1$ and $L_2$ such that the associated Markov shifts are not conjugate, yet $H(L_1) \cong H(L_2)$. The proof relies on the characterization given in the last section which allows us to choose distinct sets $\mcO_1$ and $\mcO_2$ of idempotents satisfying axioms (O1)--(O5) from a single inverse semigroup $H$. The resulting Markov shifts are easily distinguished by their entropies.

Consider the Markov subshift generated by the following transition matrix with alphabet $A_1=\{a,b,c$\}: 
\[T_1 = 
\kbordermatrix{
  	  & a & b & c \\
    a & 1 & 1 & 1 \\
    b & 1 & 0 & 1 \\
    c & 0 & 1 & 0 
  }
\]

Let $H = H(L_1)$ be the associated inverse hull, where $L_1$ is the language of the shift. Below is the top of the semilattice of $H$, with $\mcO_1$ indicated in orange text. 
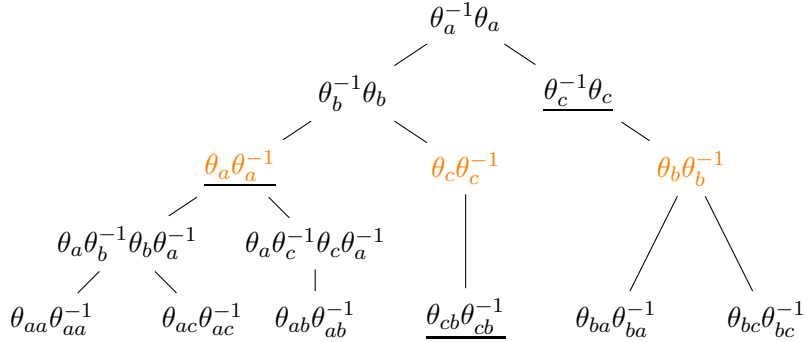
\begin{center} \begin{tikzpicture}[scale=1]
\node (10) at (0,2) {$\theta_a^{-1}\theta_a$};
\node (20) at (-1.5,1) {$\theta_b^{-1}\theta_b$};
\node (21) at (1.5,1) {\underline{$\theta_c^{-1}\theta_c$}};
\node(30) at (-3, 0) {\underline{\textcolor{orange}{$\theta_a\theta_a^{-1}$}}};
\node (31) at (0, 0) {\textcolor{orange}{$\theta_c\theta_c^{-1}$}};
\node (32) at (3, 0) {\textcolor{orange}{$\theta_b\theta_b^{-1}$}};
\node (40) at (-4.5, -1) {$\theta_a\theta_b^{-1}\theta_b\theta_a^{-1}$};
\node (41) at (-2, -1) {$\theta_a\theta_c^{-1}\theta_c\theta_a^{-1}$};
\node (50) at (-5.5, -2) {$\theta_{aa}\theta_{aa}^{-1}$};
\node (51) at (-3.5, -2) {$\theta_{ac}\theta_{ac}^{-1}$};
\node (52) at (-2, -2) {$\theta_{ab}\theta_{ab}^{-1}$};
\node (53) at (0, -2) {\underline{$\theta_{cb}\theta_{cb}^{-1}$}};
\node (54) at (2, -2) {$\theta_{ba}\theta_{ba}^{-1}$};
\node (55) at (4, -2) {$\theta_{bc}\theta_{bc}^{-1}$};
\draw (21) -- (10) -- (20);
\draw (30) -- (20) -- (31);
\draw (21) -- (32);
\draw (40) -- (30) -- (41);
\draw (50) -- (40) -- (51); 
\draw (41) -- (52);
\draw (31) -- (53);
\draw (54) -- (32) -- (55);
\end{tikzpicture}
\end{center}

Next, let $\mcO_2 =  \{ \theta_a \theta_a^{-1}, \theta_{cb}\theta_{cb}^{-1}, \theta_{c}^{-1}\theta_{c} \}$, the set of underlined idempotents in the above figure. We can verify that $(H,\mcO_2)$ satisfies the axioms of Theorem \ref{thm:main}. The elements of $\mcO_2$ are mutually orthogonal. For example, $\theta_a \theta_{a}^{-1} \theta_c^{-1} \theta_c \leq \theta_a \theta_{a}^{-1} \theta_b \theta_{b}^{-1} = 0$. For (O2), the fact that every idempotent in $H$ is comparable to at least one element in $\mcO_1$ can be used to verify that the same property holds for $\mcO_2$. For (O3), we must show that both ${(\mathcal{O}_2^{\uparrow}-\mathcal{O}_2) \cup \{0\}}$ and $\mathcal{O}_2^{\uparrow} \cup \{0\}$ are closed under multiplication. For 
\[
(\mathcal{O}_2^{\uparrow}-\mathcal{O}_2) \cup \{0\} = \{\theta_b^{-1}\theta_b,\theta_a^{-1}\theta_a, \theta_c\theta_c^{-1}, 0\}
\] 
it suffices to check the three products: 
\[
\theta_a^{-1}\theta_a \theta_b^{-1}\theta_b = \theta_b^{-1}\theta_b, \; \theta_a^{-1}\theta_a \theta_c \theta_c^{-1} = \theta_c \theta_c^{-1}, \,\text{and } \theta_b^{-1}\theta_b \theta_c \theta_c^{-1} = \theta_c \theta_c^{-1}. 
\]

Similarly, $\mathcal{O}_2^{\uparrow} \cup \{0\}$ is closed under multiplication. Next, we check property (O4). Notice that $\theta_a^{-1}\theta_a$ lies above the set $\{ \theta_a \theta_a^{-1}, \theta_{cb}\theta_{cb}^{-1}, \theta_{c}^{-1}\theta_{c} \}$, $\theta_b^{-1}\theta_b$ lies above $\{ \theta_a \theta_a^{-1}, \theta_{cb}\theta_{cb}^{-1}\}$, and $\theta_c \theta_c^{-1}$ lies above $\{ \theta_{cb}\theta_{cb}^{-1}\}$. Since these three sets are distinct, (O4) is satisfied. We know that $\theta_a^{-1}\theta_a$, $\theta_b^{-1}\theta_b$ and $\theta_c^{-1}\theta_c$ are in distinct $\mathcal{D}$-classes by Proposition \ref{prop:Dclass}. Since $\theta_c \theta_c^{-1}
 \mcD \theta_c^{-1}\theta_c$, we see that (O5) is satisfied. The last property we must check is that the language $L_2$ defined by $\mcO_2$ generates $H$ as an inverse semigroup. For this it suffices to recover the generators $\theta_a, \theta_b, \theta_c$ as products involving letters in $A_2 = \{ \theta_a, \theta_{cb}, \theta_c^{-1} \}$ and their inverses. Notice that $\theta_b = \theta_c^{-1} \theta_{cb}$ and $\theta_c = (\theta_c^{-1})^{-1}$.
 
Thus by Thereom \ref{thm:main} we have $H = H(L_1) \cong H(L_2)$. We may recover the Markov transition matrix for $L_2$ using the method given just before Theorem \ref{thm:language}. For convenience let, $x = \theta_a, y = \theta_{cb},$ and $z = \theta_c^{-1}$. We find the matrix to be 
\[T_2 = 
\kbordermatrix{
  	  & x & y & z \\
    x & 1 & 1 & 1 \\
    y & 1 & 1 & 0 \\
    z & 0 & 1 & 0 
  }
\]
The dominant eigenvalues of $T_1$ and $T_2$ are $2$ and approximately $2.206$, respectively. Therefore, by \cite[Theorem 4.3.1]{LindMarcus}, the associated Markov shifts have distinct entropies and are not conjugate.

Conversely, we can quickly verify that conjugate shifts need not have isomorphic inverse hulls. For example, consider the shifts associated with the following transition matrices
\[
  \kbordermatrix{
  	  & a & b & c \\
    a & 1 & 1 & 0 \\
    b & 0 & 0 & 1 \\
    c & 1 & 1 & 1 
  }
  \quad \quad
   \kbordermatrix{
  	  & x & y \\
    x & 1 & 1 \\
    y & 1 & 1  
  }
\]
which are shown to define conjugate shifts in Example 7.2.2 of \cite{LindMarcus}.

One can verify by inspection that the semilattices of the associated inverse hulls are distinct:
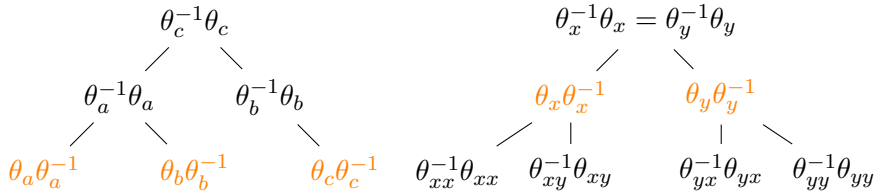
\begin{center} \begin{tikzpicture}[scale=1]
\node (10) at (-3,2) {$\theta_c^{-1}\theta_c$};
\node (11) at (3,2) {$\theta_x^{-1}\theta_x=\theta_y^{-1}\theta_y$};
\node (20) at (-4,1) {$\theta_a^{-1}\theta_a$};
\node (21) at (-2,1) {$\theta_b^{-1}\theta_b$};
\node (22) at (2,1) {\textcolor{orange}{$\theta_x\theta_x^{-1}$}};
\node (23) at (4,1) {\textcolor{orange}{$\theta_y\theta_y^{-1}$}};
\node (30) at (-5,0) {\textcolor{orange}{$\theta_a\theta_a^{-1}$}};
\node (31) at (-3,0) {\textcolor{orange}{$\theta_b\theta_b^{-1}$}};
\node (32) at (-1,0) {\textcolor{orange}{$\theta_c\theta_c^{-1}$}};
\node (33) at (.5,0) {$\theta_{xx}^{-1}\theta_{xx}$};
\node (34) at (2,0) {$\theta_{xy}^{-1}\theta_{xy}$};
\node (35) at (4,0) {$\theta_{yx}^{-1}\theta_{yx}$};
\node (36) at (5.5,0) {$\theta_{yy}^{-1}\theta_{yy}$};
\draw (20) -- (10) -- (21);
\draw (22) -- (11) -- (23);
\draw (30) -- (20) -- (31);
\draw (32) -- (21);
\draw (33) -- (22) -- (34);
\draw (35) -- (23) -- (36);
\end{tikzpicture}
\end{center}

In fact, after examining many such examples, we conjecture that isomorphic inverse hulls of Markov shifts must have associated alphabets that are the same size.

\bibliographystyle{amsplain}
\bibliography{SemigroupBib.bib}
\end{document}